\documentclass[a4paper,twoside]{article}
\usepackage{a4}
\usepackage{amssymb}
\usepackage{amsmath}
\usepackage{upref}
\usepackage{xypic}
\usepackage[dvipsnames]{color}
\makeatletter
\def\@biblabel#1{#1.}
\makeatother
\makeatletter
\let\Thebibliography=\thebibliography
\renewcommand{\thebibliography}[1]{\def\@mkboth##1##2{}\Thebibliography{#1}
\addcontentsline{toc}{section}{References}
\frenchspacing % Maybe not needed
\setlength{\@topsep}{0pt}% Delete if extra space before list
\setlength{\itemsep}{0pt}%
\setlength{\parskip}{0pt plus 2pt}%
}
\makeatother
\makeatletter
\def\mdots@{\mathinner.\nonscript\!.%
 \ifx\next,.\else\ifx\next;.\else\ifx\next..\else
 \nonscript\!\mathinner.\fi\fi\fi}
\let\ldots\mdots@
\let\cdots\mdots@
\let\dotso\mdots@
\let\dotsb\mdots@
\let\dotsm\mdots@
\let\dotsc\mdots@
\def\vdots{\vbox{\baselineskip2.8\p@ \lineskiplimit\z@
    \kern6\p@\hbox{.}\hbox{.}\hbox{.}\kern3\p@}}
\def\ddots{\mathinner{\mkern1mu\raise8.6\p@\vbox{\kern7\p@\hbox{.}}%
    \raise5.8\p@\hbox{.}\raise3\p@\hbox{.}\mkern1mu}}
\makeatother
\makeatletter
\def\@seccntformat#1{\csname the#1\endcsname.\quad}
\makeatother
\makeatletter
\renewcommand\section{\@startsection {section}{1}{\z@}%
                                   {-3.5ex \@plus -1ex \@minus -.2ex}%
                                   {2.3ex \@plus.2ex}%
                                   {\normalfont\Large\bfseries\boldmath}}
\renewcommand\subsection{\@startsection{subsection}{2}{\z@}%
                                     {-3.25ex\@plus -1ex \@minus -.2ex}%
                                     {1.5ex \@plus .2ex}%
                                     {\normalfont\large\bfseries\boldmath}}
\renewcommand\subsubsection{\@startsection{subsubsection}{3}{\z@}%
                                     {-3.25ex\@plus -1ex \@minus -.2ex}%
                                     {1.5ex \@plus .2ex}%
                                     {\normalfont\normalsize\bfseries\boldmath}}
\renewcommand\paragraph{\@startsection{paragraph}{4}{\z@}%
                                    {3.25ex \@plus1ex \@minus.2ex}%
                                    {-1em}%
                                    {\normalfont\normalsize\bfseries\boldmath}}
\renewcommand\subparagraph{\@startsection{subparagraph}{5}{\parindent}%
                                       {3.25ex \@plus1ex \@minus .2ex}%
                                       {-1em}%
                                      {\normalfont\normalsize\bfseries\boldmath}}
\makeatother
\newcommand{\art}[6]{{\sc #1, \rm #2, \it #3 \bf #4 \rm (#5), \mbox{#6}.}}
\newcommand{\book}[3]{{\sc #1, \it #2, \rm #3.}}
\newcommand{\AND}{{\rm and }}
\newcommand{\artin}[3]{{\sc #1, \rm #2,  in #3.}}
\newcommand{\artprep}[3]{{\sc #1, \rm #2, #3.}}

\RequirePackage{amsthm}
\newtheoremstyle{descriptive}%
  {\topsep}   %{\medskipamount}          % Space above
  {\topsep}   %  {\medskipamount}          % Space below
  {\rmfamily} % Body font
  {}          % Indent
  {\bfseries} % Head font
  {.}         % Punctuation after thm head
  { }         % Space after thm head
  {}          % Thm head spec(?)
\newtheoremstyle{propositional}%
  {\topsep}   %  {\medskipamount}          % Space above
  {\topsep}   %  {\medskipamount}          % Space below
  {\itshape}  % Body font
  {}          % Indent
  {\bfseries} % Head font
  {.}         % Punctuation after thm head
  { }         % Space after thm head
  {}          % Thm head spec(?)
\newtheoremstyle{remarkstyle}%
  {\topsep}   %  {\medskipamount}          % Space above
  {\topsep}   %  {\medskipamount}          % Space below
  {\rmfamily}  % Body font
  {}          % Indent
  {\itshape} % Head font
  {.}         % Punctuation after thm head
  { }         % Space after thm head
  {}          % Thm head spec(?)
\theoremstyle{propositional}
\newtheorem{thm}{Theorem}[section]
\newtheorem{lem}[thm]{Lemma}

\newtheorem{cor}[thm]{Corollary}
\theoremstyle{descriptive}
\newtheorem{deff}[thm]{Definition}
\theoremstyle{remarkstyle}
\newtheorem{remark}[thm]{Remark}
\newtheorem{example}[thm]{Example}
\makeatletter
\renewenvironment{proof}[1][\proofname]{\par
  \pushQED{\qed}%
  \normalfont
  \trivlist
  \item[\hskip\labelsep
        \itshape
    #1\@addpunct{.}]\ignorespaces
}{%
  \popQED\endtrivlist\@endpefalse
}
\makeatother

{\catcode`p =12 \catcode`t =12 \gdef\eeaa#1pt{#1}}      % Get slantfactor
\def\accentadjtext#1{\setbox0\hbox{$#1$}\kern   % Convert it with height
                \expandafter\eeaa\the\fontdimen1\textfont1 \ht0 }
\def\accentadjscript#1{\setbox0\hbox{$#1$}\kern % Convert it with height
                \expandafter\eeaa\the\fontdimen1\scriptfont1 \ht0 }
\def\accentadjscriptscript#1{\setbox0\hbox{$#1$}\kern   % Convert it with height
                \expandafter\eeaa\the\fontdimen1\scriptscriptfont1 \ht0 }
\def\accentadjtextback#1{\setbox0\hbox{$#1$}\kern       % Convert it with height
                -\expandafter\eeaa\the\fontdimen1\textfont1 \ht0 }
\def\accentadjscriptback#1{\setbox0\hbox{$#1$}\kern     % Convert it with height
                -\expandafter\eeaa\the\fontdimen1\scriptfont1 \ht0 }
\def\accentadjscriptscriptback#1{\setbox0\hbox{$#1$}\kern % Convert it with height
                -\expandafter\eeaa\the\fontdimen1\scriptscriptfont1 \ht0 }
\def\itoverline#1{{\mathsurround0pt\mathchoice
        {\rlap{$\accentadjtext{\displaystyle #1}
                \accentadjtext{\vrule height1.593pt}
                \overline{\phantom{\displaystyle #1}
                \accentadjtextback{\displaystyle #1}}$}{#1}}
        {\rlap{$\accentadjtext{\textstyle #1}
                \accentadjtext{\vrule height1.593pt}
                \overline{\phantom{\textstyle #1}
                \accentadjtextback{\textstyle #1}}$}{#1}}
        {\rlap{$\accentadjscript{\scriptstyle #1}
                \accentadjscript{\vrule height1.593pt}
                \overline{\phantom{\scriptstyle #1}
                \accentadjscriptback{\scriptstyle #1}}$}{#1}}
        {\rlap{$\accentadjscriptscript{\scriptscriptstyle #1}
                \accentadjscriptscript{\vrule height1.593pt}
                \overline{\phantom{\scriptscriptstyle #1}
                \accentadjscriptscriptback{\scriptscriptstyle #1}}$}{#1}}}}
\newcommand{\setm}{\setminus}
\newcommand{\Cp}{{C_p}}
\newcommand{\grad}{\nabla}
\DeclareMathOperator{\Div}{div}
\DeclareMathOperator{\cp}{cap}
\newcommand{\capp}{\cp_p}
\DeclareMathOperator{\BMO}{BMO}
\newcommand{\bdry}{\partial}
\newcommand{\loc}{_{\rm loc}}
\newcommand{\limplus}{{\mathchoice{\raise.17ex\hbox{$\scriptstyle +$}}
                {\raise.17ex\hbox{$\scriptstyle +$}}
                {\raise.1ex\hbox{$\scriptscriptstyle +$}}
                {\scriptscriptstyle +}}}
\newcommand{\alp}{\alpha}
\newcommand{\A}{{\cal A}}
\newcommand{\al}{\alpha}
\newcommand{\albar}{\itoverline{\al}}
\newcommand{\be}{\beta}
\newcommand{\bebar}{\bar{\be}}
\newcommand{\de}{\delta}
\newcommand{\ga}{\gamma}
\newcommand{\Om}{\Omega}
\newcommand{\om}{\omega}
\renewcommand{\phi}{\varphi}
\newcommand{\eps}{\varepsilon}
\newcommand{\p}{{$p\mspace{1mu}$}}
\newcommand{\R}{\mathbf{R}}
\newcommand{\K}{\widetilde{\cal K}}
\newcommand{\Wp}{W^{1,p}}
\newcommand{\Np}{W^{1,p}}
\def\cprime{{\mathsurround0pt$'$}}
\numberwithin{equation}{section}
\newenvironment{ack}{\medskip{\it Acknowledgement.}}{}

\begin{document}

\author{Jana Bj\"orn}
\title{Sharp exponents and a Wiener type condition for boundary 
regularity of quasiminimizers}
\author{
Jana Bj\"orn \\
\it\small Department of Mathematics, Link\"oping University, \\
\it\small SE-581 83 Link\"oping, Sweden\/{\rm ;}
\it \small jana.bjorn@liu.se
}

\date{}
\maketitle

\noindent{\small {\bf Abstract}. 
We obtain a sufficient condition for boundary regularity of quasiminimizers
of the \p-energy integral in terms of a Wiener type sum of power type. 
The exponent in the sum is independent of the dimension and is
explicitly expressed in terms of $p$ and the quasiminimizing constant. 
We also show by an example that the exponent is sharp in a certain sense.

}



\section{Introduction}

Let $\Om\subset\R^n$ be a bounded open set and  
consider the  \emph{Dirichlet problem}
of finding a \p-harmonic function $u$, i.e.\  a solution of the \p-Laplace
equation 
\[
\Delta_pu:=\Div(|\grad u|^{p-2}\grad u)=0, \quad 1<p<\infty,
\] 
in $\Om$ with prescribed boundary values $f$.
Even if the boundary data are continuous, it cannot be guaranteed in general
that the solution attains its boundary values as limits  
\begin{equation}   \label{eq-def-regular}
\lim_{\Om\ni x\to x_0} u(x) = f(x_0)
\end{equation}
at all boundary points $x_0\in\bdry\Om$.
If \eqref{eq-def-regular} holds for every continuous $f$ then $x_0$ is
called \emph{regular}. 

The classical Wiener criterion asserts that the regularity of 
a boundary point $x_0\in\bdry\Om$ is equivalent to 
\begin{equation}
\int_0^1 \biggl( \frac{\capp(B(x_0,\rho)\setm\Om,B(x_0,2\rho))}
      {\rho^{n-p}}  
\biggr)^{1/(p-1)} \frac{d\rho}{\rho}=\infty
\label{eq-wiener-crit}
\end{equation}
where $\capp$ is the variational capacity and $B(x_0,\rho)$ denotes the ball
with centre $x_0$ and radius $\rho$.
The Wiener criterion was proved by Wiener~\cite{Wiener24} in 1924 in the linear case $p=2$ 
(i.e.\ for harmonic functions).
In the nonlinear case, for general $p>1$, the sufficiency part was obtained by 
Maz\cprime ya~\cite{Mazya70} in 1970
and it then took more than 20 years for the necessity part, due to
Kilpel\"ainen--Mal\'y~\cite{KilMa}, even though for 
$p=n$ it was obtained already by Lindqvist--Martio~\cite{LiMa85} in 1985.
Note that at those times it was not even clear that the exponent
should be $1/(p-1)$. 

It is well known that \p-harmonic functions are minimizers of the \p-energy,
i.e.\  that the above solution $u$ of the Dirichlet problem satisfies
\begin{equation}  \label{eq-def-minimizer-intro}
\int_\Om |\grad u|^p\,dx \le \int_\Om |\grad (u+\phi)|^p\,dx
\end{equation}
for all test functions $\phi\in C^\infty_0(\Om)$.
Solutions of more general equations, such as $\Div \A(x,\grad u)=0$
with $a_1 |\xi|^p \le \A(x,\xi)\cdot \xi \le a_2 |\xi|^p$,
do not satisfy~\eqref{eq-def-minimizer-intro} in this form 
but it can be verified that they are \emph{quasiminimizers}, 
namely that there is a constant $Q\ge1$ such that
\begin{equation}  \label{eq-def-q-min-intro}
\int_{\phi\ne0} |\grad u|^p\,dx \le Q \int_{\phi\ne0} |\grad (u+\phi)|^p\,dx.
\end{equation}

Quasiminimizers were introduced by
Giaquinta and Giusti~\cite{GG1},~\cite{GG2} as a tool for a unified
treatment of variational integrals, elliptic equations and
quasiregular mappings on $\R^n$. 
They have since then been studied by various authors and it has turned out
that they share many (though not all) properties with \p-harmonic functions.
In particular, De Giorgi's method applies to quasiminimizers and shows that
they are locally H\"older continuous~\cite{GG2}.  
The Harnack inequality for quasiminimizers was proved by
DiBenedetto--Trudinger~\cite{DiBTru} and Mal\'y~\cite{Maly}, where also
the strong maximum principle was obtained.

Tolksdorf~\cite{tolksdorf} obtained a Caccioppoli inequality and
a convexity result for quasiminimizers.
Further regularity results in $\R^n$ can be found in 
Latvala~\cite{LatBMO}
and Kinnunen--Kotilainen--Latvala~\cite{KiKoLa}. 
An obstacle problem for quasiminimizers was considered by Ivert~\cite{Ivert}.
Recently, it was discovered that quasiminimizers include solutions 
of even larger classes of equations, such as Ricatti equations,
see Martio~\cite{MarRic1},~\cite{MarRic2}. 

Compared with
the theory of \p-harmonic functions there is no common  differential equation
for quasiminimizers to work with,
only the variational inequality can be used.  There is also no
comparison principle nor uniqueness for the Dirichlet problem.
On the other hand, quasiminimizers are more flexible than \p-harmonic
functions and are preserved by quasiregular mappings,
as shown by Korte--Marola--Shanmugalingam~\cite{KoMaSh}.
Potential theory for quasiminimizers was developed in 
Kinnunen--Martio~\cite{KiMa03}.

Unlike \p-harmonic functions, which on $\R$ reduce to linear functions,
quasiminimizers also have a rich 1-dimensional theory, as seen in
Martio--Sbordone~\cite{MaSb}, Judin~\cite{judin}, 
Martio~\cite{martioReflect},~\cite{Martio},
Uppman~\cite{uppman}, Bj\"orn--Bj\"orn~\cite{BB-power} 
and Bj\"orn--Bj\"orn--Korte~\cite{Riikka}.

Boundary regularity for quasiminimizers was studied by 
Ziemer~\cite{Ziem86} in $\R^n$ and by Bj\"orn~\cite{Bj02} in metric spaces,
where explicit pointwise estimates were also given.
A weak Kellogg property and several other qualitative results about boundary
regularity for quasiminimizers were obtained by 
A.~Bj\"orn~\cite{ABkellogg},~\cite{ABclass} and~\cite{ABcluster}.
It was also shown by A.~Bj\"orn--Martio~\cite[Theorem~6.2]{BjMa-paste}
that  regularity for quasiminimizers is a local property, i.e.\ that 
it only depends on the geometry of $\Om$ in a neighbourhood of $x_0$.
For \p-harmonic functions, this follows from the Wiener criterion, but
for quasiminimizers the exact form of a Wiener type condition is not known
and there is a substantial gap between the known sufficient and necessary
conditions.

During the last 15 years, quasiminimizers have also been studied on metric
measure spaces.
Interior regularity and qualitative properties of quasiminimizers 
in metric spaces have been
studied in \cite{ABremove},  \cite{Riikka}, \cite{BBM}, \cite{BMarola}, 
\cite{KiMaMa}, \cite{KiMa03} and \cite{KiSh01}. 
Various results on their boundary behaviour were obtained in 
\cite{ABkellogg}, \cite{ABclass}, \cite{ABcluster}, \cite{BjMa-paste},  
\cite{Bj02} and \cite{JBCalcVar}.

In Ziemer~\cite{Ziem86} and Bj\"orn~\cite{Bj02} 
it was shown that the divergence
 of certain integrals (or sums) similar 
to~\eqref{eq-wiener-crit} is sufficient for boundary regularity.
However, the integrand in those conditions is an exponential function
decaying much faster than the power in the classical Wiener 
criterion~\eqref{eq-wiener-crit}. 
This makes it more difficult for the sums to diverge and the conditions 
are therefore much more restrictive.
Nevertheless, they guarantee regularity e.g.\ if the complement of $\Om$
has a cork-screw or if $\Om$ is porous at $x_0$.

In a very recent preprint~\cite{DiBen-Gia} DiBenedetto and Gianazza
use weak  Harnack
inequalities near the boundary to obtain a sufficient condition of
power-type with some exponent, depending only on $p$, $Q$ and $n$, 
which 
is traceable through their calculations.

In this paper we obtain a sufficient condition
for boundary regularity of quasiminimizers,
which is more similar to the Wiener criterion 
and whose exponent is explicit and independent of $n$.
More precisely, we prove the following result.

\begin{thm}   \label{thm-regular-intro}
Let $\Om\subset\R^n$ be bounded and open, $x_0\in\bdry \Om$ and $Q>1$.
Assume that for some $\eps>0$
the Wiener type sum
\begin{equation}
\sum_{j=0}^{\infty} \biggl(
    \frac{\capp(B(x_0,2^{-j-1})\setm\Om,B(x_0,2^{-j}))}
      {2^{-j(n-p)}}    
\biggr)^{\albar/(p-1)+\eps} =\infty,
\label{eq-div-wiener-sum-intro}
\end{equation}
where $\albar\ge1$ is the unique solution in $ [1,\infty)$ of the equation
\begin{equation*}   
Q=\frac{\albar^p}{1+p(\albar-1)}.
\end{equation*}
Then $x_0$ is a regular boundary point for $Q$-quasiminimizers.
\end{thm}

Note that $\albar\to1$ as $Q\to1$.
We also remark that the question of regularity is only interesting 
for $p\le n$, 
since for $p>n$ every point is regular (because of Sobolev embeddings)
and the Wiener type sum always diverges.

It is clear that the sum in~\eqref{eq-div-wiener-sum-intro} can equivalently 
be replaced by an integral. 
Since $\albar$ satisfies the estimate $Q^{1/(p-1)}\le\albar<(pQ)^{1/(p-1)}$,
our result in particular means that the explicit condition
\[
\int_0^1 \biggl( \frac{\capp(B(x_0,\rho)\setm\Om,B(x_0,2\rho))}
      {\rho^{n-p}}  
\biggr)^{\frac{(pQ)^{1/(p-1)}}{p-1}} \frac{d\rho}{\rho}=\infty
\]
guarantees the boundary regularity of $x_0$ for $Q$-quasiminimizers in $\Om$.
Note however that, unlike $\albar/(p-1)+\eps$ 
in~\eqref{eq-div-wiener-sum-intro},
the exponent $(pQ)^{1/(p-1)}/(p-1)$ does not have the correct asymptotics
as $Q\to1$.
We also show that in some sense the exponent $\albar/(p-1)+\eps$ is sharp,
possibly up to $\eps$.

Our proof is based on the capacitary estimates for quasiminimizing potentials
from Martio~\cite{Martio}.
Other important tools will be the obstacle problem and a pasting lemma from 
A.~Bj\"orn--Martio~\cite{BjMa-paste}.
It also uses the explicit examples of power-type
quasiminimizers from Judin~\cite{judin}, Martio~\cite{martioReflect}
and Bj\"orn--Bj\"orn~\cite{BB-power}, as well as
the optimal quasiminimizing constants $Q$ for the
powers $|x|^\al$ in one and several dimensions, obtained therein. 
As a byproduct of our investigations, we also show that there is a one-to-one 
correspondence between the powers $|x|^\al$ associated with $Q$ and $p$
and those associated with the ``dual'' constants $Q^{1/(p-1)}$ and $p/(p-1)$.

The outline of the paper is as follows.
In Section~\ref{sect-prelim} we recall some definitions 
and properties of quasiminimizers.
Theorem~\ref{thm-regular-intro} is proved in the subsequent sections.
First, we prove a preliminary version in Section~\ref{sect-Wiener-est},
which we simplify to the above form in Section~\ref{sect-simplify}.
Finally, Section~\ref{sect-sharpness} is devoted to 
demonstrating the sharpness of the obtained exponent.

\begin{ack}
The author has been supported by the Swedish Research Council.
She also thanks Olli Martio for sending her the preprint version
of~\cite{Martio} and for useful discussions.
\end{ack}

\section{Preliminaries and auxiliary results}
\label{sect-prelim}

Throughout the rest of the
paper we assume that $\Om\subset\R^n$ is a bounded open set 
and that $1<p\le n$ and $Q>1$ are fixed.

The following is one of several equivalent definitions 
of quasi(super/sub)mi\-ni\-m\-izers,
see A.~Bj\"orn~\cite{ABremove}.
Recall that the Sobolev space $\Wp(\Om)$ consists of all 
$L^p$-functions in $\Om$ with distributional gradients in $L^p$.
The space $\Wp_0(\Om)$ is the subspace of $\Wp(\Om)$ with zero boundary
values. 

\begin{deff}
A function $u\in\Wp\loc(\Om)$
is a $Q$\emph{-quasiminimizer} in $\Om$ if~\eqref{eq-def-q-min-intro} 
holds for all $\phi\in\Wp_0(\Om)$.
If $u\in\Wp\loc(\Om)$ and~\eqref{eq-def-q-min-intro} 
holds for all nonnegative (nonpositive) $\phi\in\Wp_0(\Om)$ then $u$
is a $Q$\emph{-quasisuper(sub)minimizer} in $\Om$.
\end{deff}

It was shown already by Giaquinta--Giusti~\cite{GG2}
that quasiminimizers (or rather their suitable representatives in 
$\Wp\loc$) are locally H\"older continuous.
Similarly, quasisuper(sub)minimizers can be shown to have lower (upper)
semicontinuous representatives.
We therefore throughout the paper consider only such representatives.

Quasiminimizers obey the maximum and minimum principles saying that
for every bounded open set $\Om'\subset\Om$ with $\overline{\Om'}\subset\Om$,
\[
\inf_{\Om'} u = \inf_{\bdry\Om'} \quad \text{and} \quad
\sup_{\Om'} u = \sup_{\bdry\Om'},
\]
see~\cite{GG2}.
On the other hand, by considering solutions of $\Div \A(x,\grad u)=0$
with different but comparable $\A$ it is easy to see that the comparison
principle, which is otherwise a useful tool for \p-harmonic functions, 
fails for quasiminimizers.
In other words, it can happen for two quasiminimizers that $u_1\le u_2$
holds on $\bdry\Om$ but fails in $\Om$.

One way of compensating for the lack of comparison principle for
quasiminimizers, is to use pasting lemmas as in
A.~Bj\"orn--Martio~\cite{BjMa-paste}.
The following pasting lemma for quasisuperminimizers
is a special case of Theorem~4.1 in~\cite{BjMa-paste}.

\begin{lem}        \label{lem-extend}
Let $u\in\Wp\loc(\Om)$ be a $Q$-quasisuperminimizer in $\Om\setm F$, where 
$F\subset\Om$ is relatively closed in $\Om$.
Assume that $u\le1$ in $\Om\setm F$ and $u=1$ on $F$.
Then $u$ is a $Q$-quasisuperminimizer in $\Om$.
\end{lem}

Here we use the refined Sobolev spaces as in 
Heinonen--Kilpel\"ainen--Martio~\cite[Section~4]{HeKiMa},
i.e.\ we consider only the quasicontinuous representatives of Sobolev functions,
which are well defined up to sets of \p-capacity zero.
This means that equalities such as $u=1$ for Sobolev functions are 
regarded as holding up to sets of \p-capacity zero.
Recall that the Sobolev \p-capacity is for a compact set $K\subset\R^n$
defined as
\[
\Cp(K)=\inf \int_{\R^n} (|\phi|^p + |\grad\phi|^p) \,dx,
\]
where the infimum is taken over all $\phi\in C^\infty_0(\R^n)$ such that
$\phi\ge1$ on $K$.
We shall also use the variational capacity $\capp$, which for a compact
set $K\subset \Om$ is defined as
\[
\capp(K,\Om)=\inf \int_{\Om} |\grad\phi|^p \,dx,
\]
where the infimum is taken over all $\phi\in C^\infty_0(\Om)$ such that
$\phi\ge1$ on $K$.

A function $u\in\Wp_0(\Om)$ is a 
$Q$\emph{-quasiminimizing potential} in $\Om$ for a relatively closed set
 $F\subset\Om$ if
it is a $Q$-quasiminimizer in $\Om\setm F$ and $u=1$ on~$F$.
It follows from the maximum principle that $0\le u\le1$.
Lemma~\ref{lem-extend} implies that
every quasiminimizing potential in $\Om$ 
is a quasisuperminimizer and thus lower semicontinuous.
It is also easily verified, and follows from Lemma~\ref{lem-extend},
that truncations $\min\{u,k\}$ of quasisuperminimizers 
are quasisuperminimizers with the same constant $Q$ for every $k\in\R$.

Another useful tool for studying and constructing quasiminimizers is
the upper obstacle problem as follows.
We will only use it with Sobolev obstacle and boundary values.
Given $f\in\Wp(\Om)$ and $\om\in\Wp\loc(\Om)$ let
\[
\K_{\om,f}(\Om) 
= \{ v\in\Wp(\Om): v-f\in\Wp_0(\Om) \text{ and } v\le \om \text{ in }\Om\}.
\]
If $\K_{\om,f}(\Om)\ne\emptyset$ then Theorem~3.12 in
Heinonen--Kilpel\"ainen--Martio~\cite{HeKiMa} implies that 
there exists a unique (in the Sobolev sense) 
$u\in\K_{\om,f}(\Om)$ such that
\[
\int_{\Om} |\grad u|^p \,dx \le \int_{\Om} |\grad v|^p \,dx
\quad \text{for all } v\in\K_{\om,f}(\Om).
\]
We call $u$ the solution of the 
$\K_{\om,f}(\Om)$\emph{-obstacle problem}
with the \emph{upper} obstacle $\om$ and the boundary values $f$.
We alert the reader that in the literature, such as~\cite{HeKiMa},  
one usually considers the obstacle problem with a lower obstacle $\psi\le u$,
or possibly a double obstacle problem.
For us, upper obstacles are more natural.

The following is essentially Lemma~6.3 in A.~Bj\"orn--Martio~\cite{BjMa-paste},
here formulated for the upper obstacle problem.
See also Lemma~6.1 in Martio~\cite{Martio} for a similar formulation.
Note that $\K_{\om,f}(\Om)\ne\emptyset$ if 
$(f-\om)_\limplus\in\Wp_0(\Om)$.

\begin{lem}   \label{lem-Olli-obst-pr}
Let $f\in\Wp(\Om)$ and let $\om$ be a
$Q$-quasi\-super\-mi\-ni\-m\-izer in $\Om$ such that 
$(f-\om)_\limplus\in\Np_0(\Om)$.
Then the solution of the obstacle problem with the upper obstacle $\om$
and the boundary data $f$ is a $Q$-quasiminimizer in $\Om$.
\end{lem}

We conclude this section by defining regular boundary points for quasiminimizers.

\begin{deff}
A point $x_0\in\bdry\Om$ is a \emph{regular boundary point} for 
$Q$-quasiminimizers
if~\eqref{eq-def-regular} holds for all $f\in C(\bdry\Om)\cap \Wp(\Om)$ and all 
$Q$-quasiminimizers $u$ with $u-f\in\Wp_0(\Om)$.
\end{deff}

\section{Wiener type estimates for quasiminimizers}
\label{sect-Wiener-est}

In this section we prove a preliminary version of Theorem~\ref{thm-regular-intro}.
It will be further improved in the next section.
The following definition from
Martio--Sbordone~\cite{MaSb} plays an important role for the estimates:
For $p,t>1$, let $p_1(p,t)$ be the unique solution in $(p,\infty)$ 
of the equation
\begin{equation}   \label{eq-def-p1}
t^p\frac{x-p}{x} \Bigl( \frac{x}{x-1} \Bigr)^p = 1.
\end{equation}
The unique solubility follows from the monotonicity of the left-hand side 
(which is easily proved by differentiation) and the fact that its
limits as $x\to p$ and $x\to\infty$ are $0$ and $t^p>1$.

The proof of Theorem~\ref{thm-regular-intro} is based on the 
following result from Martio~\cite[Corollary~5.3]{Martio}.
(Note that by the minimum principle, $\inf_{2B}u=\inf_{\bdry(2B)}u$.)

\begin{thm}   \label{thm-Olli}
Let $B=B(x_0,r)\subset\R^n$
and let  $u$ be a $Q$-quasiminimizing potential of a compact set 
$K\subset \itoverline{B}$ in $3B:=B(x_0,3r)$. 
Then 
\[
\inf_{2B}u \ge c \biggl(
    \frac{\capp(K,3B)}{r^{n-p}} \biggr)^{1/\de}
\]
where 
\[
\de = p-\frac{s}{s-1}>0, \quad
s\in\biggl( \frac{p}{p-1}, p_1\biggl(\frac{p}{p-1},Q^{1/p}\biggr)\biggr)
\] 
is arbitrary and $c>0$ depends only on $n$, $p$, $Q$ and $\de$.
\end{thm}

An iteration of Theorem~\ref{thm-Olli} now makes it possible to prove the
following estimate.

\begin{thm} \label{thm-iter-pot-est}
Let $x_0\in\R^n$, $r>0$ and $B_j=B(x_0,r_j)$, 
where $r_j=3^{-j}r$, $j=0,1,\ldots$.
Let $u$ be a $Q$-quasiminimizing
potential for a compact set $K\subset\itoverline{B}_1$ in $B_0$.
Then for all $k=0,1,\ldots,$
\begin{equation}   \label{eq-exp-est}
\inf_{B_{k+1}}u \ge 1 - \exp \biggl( -c \sum_{j=0}^{k} \biggl(
    \frac{\capp(K\cap \itoverline{B}_{j+1},B_{j})}
      {r_j^{n-p}}  
\biggr)^{1/\de} \biggr),
\end{equation}
where $\de$ and $c$ are as in Theorem~\ref{thm-Olli}.
\end{thm}

\begin{remark}
Since $e^{-t}$ is essentially equal to $1-t$ for small $t$, the above estimate 
can also be written as
\begin{equation*}  
\inf_{B_{k+1}}u \ge c' \sum_{j=0}^{k} \biggl(
    \frac{\capp(K\cap \itoverline{B}_{j+1},B_{j})}
      {r_j^{n-p}}  
\biggr)^{1/\de}.
\end{equation*}
An opposite estimate from above, with the (nonoptimal)
exponent $1/p$ instead of $1/\de$
was proved in Bj\"orn~\cite[Theorem~3.6]{JBCalcVar},
with $c'$ depending on $n$, $p$ and $Q$.
\end{remark}

\begin{proof}[Proof of Theorem~\ref{thm-iter-pot-est}.]
Let $m_1 = \inf_{B_1} u$.
By Theorem~\ref{thm-Olli} we have
\[
m_1 \ge \inf_{2B_1} u \ge
c \biggl( \frac{\capp(K,B_{0})}{r_0^{n-p}}   
\biggr)^{1/\de}.
\]
As $1+t\le e^t$ for $t\in\R$, this implies 
\begin{equation}
1-m_1 \le 1-c \biggl( \frac{\capp(K,B_{0})}{r_0^{n-p}} \biggr)^{1/\de}
     \le \exp \biggl( -c \biggl( \frac{\capp(K,B_{0})}
             {r_0^{n-p}}  
\biggr)^{1/\de} \biggr).
\label{eq-est-m1-with-exp}
\end{equation}
Next, let $D_1=B_1\setm(K\cap\itoverline{B}_2)$ and 
let $f_1$ be a Lipschitz function
in $B_0$ such that
\[
f_1 = \left\{ \begin{array}{ll}
              m_1 & \text{on } \bdry B_1,\\
              1   & \text{on } \itoverline{B}_2. 
              \end{array} \right.
\]
Let $u_1$ be the solution of the obstacle problem in $D_1$ 
with the upper obstacle $u$ and the boundary values $f_1$.
By Lemma~\ref{lem-Olli-obst-pr}, $u_1$ is a $Q$-quasiminimizer 
in $D_1$ and Lemma~\ref{lem-extend} shows that $u_1$, 
when extended by 1 in $K\cap\itoverline{B}_2$, 
is a $Q$-quasisuperminimizer in $B_1$.

Theorem~\ref{thm-Olli} with $K$, $B(x_0,3R)$ and $u$
replaced by $K\cap\itoverline{B}_2$, $B_1$ and $(u_1-m_1)/(1-m_1)$,
respectively, shows that
\[
m_2:= \inf_{B_2} u_1 \ge
  m_1 + c(1-m_1) \biggl( \frac{\capp(K\cap\itoverline{B}_2,B_1)}
         {r_1^{n-p}}  
\biggr)^{1/\de},
\]
and consequently,
\[
1-m_2 \le (1-m_1)  \exp 
\biggl( -c \biggl( \frac{\capp(K\cap\itoverline{B}_2,B_1)}
             {r_1^{n-p}}  
\biggr)^{1/\de} \biggr).
\]
We continue in this way, putting
$D_j=B_j\setm(K\cap\itoverline{B}_{j+1})$,
\[
f_j = \left\{ \begin{array}{ll}
              m_j & \text{on } \bdry B_j,\\
              1   & \text{on } \itoverline{B}_{j+1},
              \end{array} \right.
\]
and letting $u_j$ be the solution of the obstacle problem in $D_j$ 
with the upper obstacle $u_{j-1}$ 
(which is a $Q$-quasisuperminimizer in $D_j$ by Lemma~\ref{lem-extend}) 
and the boundary values $f_j$, $j=2,3,\ldots.$
Another application of Theorem~\ref{thm-Olli} implies that
\[
m_{j+1}:= \inf_{B_{j+1}} u_j \ge m_j + c(1-m_j)
         \biggl( \frac{\capp(K\cap\itoverline{B}_{j+1},B_j)}
         {r_j^{n-p}}  
\biggr)^{1/\de},
\quad j=2,3,\ldots,
\]
and consequently,
\[
1-m_{j+1} \le (1-m_j)  \exp \biggl( -c \biggl( 
    \frac{\capp(K\cap\itoverline{B}_{j+1},B_j)}
             {r_j^{n-p}}  
\biggr)^{1/\de} \biggr).
\]
Iterating this inequality and using \eqref{eq-est-m1-with-exp} we obtain 
for $k=1,2,\ldots,$
\[
1-m_{k+1} \le \exp \biggl( -c \sum_{j=0}^{k} \biggl(
    \frac{\capp(K\cap \itoverline{B}_{j+1},B_{j})}
      {r_j^{n-p}}   
\biggr)^{1/\de} \biggr).
\]
As $u\ge u_1$ in $B_1$ and $u_j \ge u_{j+1}$ in $B_{j+1}$, $j=1,2,\ldots$,
this finishes the proof.
\end{proof}

\begin{thm}   \label{thm-regular}
Let $x_0\in\bdry \Om$, $r>0$ and $B_j=B(x_0,r_j)$, where $r_j=3^{-j}r$, 
$j=0,1,\ldots.$
Assume that for some (or equivalently for all) $r>0$,
\begin{equation}
\sum_{j=0}^{\infty} \biggl(
    \frac{\capp(B_{j+1}\setm\Om,B_{j})}
      {r_j^{n-p}} \biggr)^{1/\de} =\infty,
\label{eq-div-wiener-sum}
\end{equation}
where $\de$ is as in Theorem~\ref{thm-Olli}.
Then $x_0$ is a regular boundary point for $Q$-quasiminimizers.
\end{thm}

\begin{proof}  
Let $f\in C(\bdry\Om)\cap\Wp(\Om)$ and $u$ be a $Q$-quasiminimimizer in $\Om$ 
such that $u-f\in\Np_0(\Om)$.
We can without loss of generality assume that $f\ge0$  
and $f(x_0)=1$.
Let $\eps>0$ and find $r>0$ such that $f\ge1-\eps$ in $B_0=B(x_0,r)$.

Let $\bar{u}=\min\{u,1-\eps\}$.
This is a $Q$-quasisuperminimizer in $B_0\cap\Om$. 
Lemma~\ref{lem-extend} 
implies that $\bar{u}$ is a $Q$-quasisuperminimizer in $B_0$ as well.
Next let $K=\itoverline{B}_1\setm\Om$ and let
$u_0\in\Np_0(B_0)$ be the solution of the obstacle problem in 
$B_0\setm K$ with the upper obstacle $\bar{u}$ and
the boundary values $0$ on $\bdry B_0$ and $1-\eps$ on $K$.
By Lemma~\ref{lem-Olli-obst-pr}, $u_0$ is a $Q$-quasiminimizer in 
$B_0\setm K$.
Theorem~\ref{thm-iter-pot-est} implies that for all $k=1,2,\ldots,$
\[
\inf_{B_{k+1}}u_0 \ge (1-\eps) \biggl(1 - \exp \biggl( -c \sum_{j=0}^{k} \biggl(
    \frac{\capp(\itoverline{B}_{j+1}\setm\Om,B_{j})}
      {r_j^{n-p}}  
\biggr)^{1/\de} \biggr) \biggr).
\]
The assumption~\eqref{eq-div-wiener-sum} then yields
\[
\liminf_{y\to x_0} u_0(y) \ge 1-\eps.
\]
As $\eps>0$ was arbitrary and $u\ge\bar{u}\ge u_0$, we obtain
\[
\liminf_{y\to x_0} u(y) \ge 1 = f(x_0).
\]
Applying the same argument to $-f$ finishes the proof.
\end{proof}

\begin{remark}
It is clear that the radii $3^{-j}r$ in Theorems~\ref{thm-iter-pot-est}
and~\ref{thm-regular}  can equivalently
be replaced by any other geometric sequence,
such as $2^{-j}$ in~\eqref{eq-div-wiener-sum-intro}, 
and that the sum can equivalently be replaced by an integral. 
\end{remark}

\section{Simplifying the exponent in Theorem~\ref{thm-regular}}
\label{sect-simplify}

In this section we shall investigate how the exponent $1/\de$ in 
Theorems~\ref{thm-Olli}, \ref{thm-iter-pot-est}
and~\ref{thm-regular} depends on $Q$ and $p$. 
We will provide a rather explicit form for it.
Recall that 
\[
\de=p-\frac{s}{s-1}, 
\quad \text{where }
s\in \biggl(\frac{p}{p-1},p_1\biggr) \quad \text{is arbitrary and}
\] 
\[
p_1=p_1\biggl(\frac{p}{p-1},Q^{1/p} \biggr)>\frac{p}{p-1}
\]
is the unique solution of the equation
\[
(Q^{1/p})^{p/(p-1)} \frac{x-\frac{p}{p-1}}{x} 
\Bigl( \frac{x}{x-1} \Bigr)^{p/(p-1)} = 1,
\]
see~\eqref{eq-def-p1}.
Our aim is to express $p_1$ more explicitly in terms of $Q$ and $p$.
This will be done in several steps. 
For this, the following identity will be crucial. 

For $p>1$ and $\al>1-1/p$, let 
\begin{equation}   \label{eq-def-Qalp}
Q(\al,p) = \frac{\al^p}{1+p(\al-1)}.
\end{equation}
The significance of~\eqref{eq-def-Qalp} is given by the following theorem,
which for $n=1$ and $p=2$ was obtained by 
Judin~\cite[Example~4.0.26 and Remark~4.0.28]{judin}
and Martio~\cite{martioReflect}, Section~5.
For general $p>n$ it is Theorem~6.1 in Bj\"orn--Bj\"orn~\cite{BB-power}.
Here we use it with $n=1$. 

\begin{thm} \label{thm-qmin-power-p>n}
Let $p >n$. 
Then $|x|^\al$ is a quasiminimizer 
in\/ $B(0,1) \setm \{0\}\subset\R^n$
if and only if $\alp > 1-n/p$ or $\alp=0$.
Moreover, if $\alp > 1-n/p$, then 
\begin{equation*} 
\biggl(\frac{p-1}{p-n}\biggr)^{p-1} \frac{\alp^p}{n+p(\alp-1)}
\end{equation*}
is the best quasiminimizer constant for $|x|^\al$.

In particular, if $\al>1-1/p$ then $x^\al$ is a quasiminimizer 
in $(0,1)$ with the best quasiminimizer constant equal to $Q(\al,p)$.
\end{thm}

Note that, given $Q>1$, there are exactly two exponents $1-1/p<\al<1<\albar$
such that $Q=Q(\al,p)=Q(\albar,p)$.
This is easily shown by differentiating~\eqref{eq-def-Qalp} and noting
that the derivative is negative for $\al<1$ and positive for $\al>1$,
and that $Q(\al,p)\to\infty$ as $\al\to 1-1/p$ and as $\al\to\infty$.
For $Q=1$ we have $\al=\albar=1$.
The following lemma is easily proved by direct calculation.

\begin{lem}  \label{lem-pQ}
Let $\al\in(1-1/p,1]$ and $Q=Q(\al,p)$.
Then $p_1(p,Q^{1/p}) = 1/(1-\al)$.
\end{lem}

Replacing $p$ and $Q$ in Lemma~\ref{lem-pQ} by $p'=p/(p-1)$ 
and $Q'=Q^{1/(p-1)}$, respectively,
we immediately obtain the following result, which will be useful when 
simplifying the exponent $1/\de$ in the Wiener type condition.
Note that $Q^{1/p}=(Q')^{1/p'}$ and that $x^\be$ is a quasiminimizer 
of the $p'$-energy in $(0,1)$ if and only if $\be=0$ or $\be>1/p$, 
by Theorem~\ref{thm-qmin-power-p>n}.

\begin{cor}  \label{cor-beta-p1}
Let $\be\in(1/p,1]$ be such that  
\[
Q^{1/(p-1)}=\frac{\be^{\frac{p}{p-1}}}{1+\frac{p}{p-1}(\be-1)}.
\]
Then 
\[
p_1\biggl(\frac{p}{p-1},Q^{1/p}\biggr) = \frac{1}{1-\be}.
\]
\end{cor}

Thus, Theorem~\ref{thm-regular} can be reformulated in terms of the 
exponent $\be$ and this will be used to prove Theorem~\ref{thm-regular-intro}.
To replace $\be$ by an exponent associated with $Q$ and $p$,
rather than $Q^{1/(p-1)}$ and $p'$, we use the following result,
which is also easily proved by direct calculation.

\begin{lem}   \label{lem-al-be}
For $p>1$ and $\al>1-1/p$, let $Q(\al,p)$ be as in~\eqref{eq-def-Qalp}
and
\[
\beta(\al) = \frac{\al}{1+p(\al-1)}.
\]
Then $Q(\al,p)^{1/(p-1)} = Q(\beta(\al),p/(p-1))$.
\end{lem}

\begin{remark}
Note that $\beta(\al)=1$ if and only if $\al=1$, and that
\[
\beta(\al) \to \left\{ \begin{array}{ll}
              \frac1p & \text{as } \al\to\infty,\\
              \infty   & \text{as } \al\to 1-\frac1p.
              \end{array} \right.
\]
Thus, Lemma~\ref{lem-al-be} provides us with an explicit one-to-one
correspondence between power-type quasiminimizers associated with $Q$ and $p$,
as in~\eqref{eq-def-Qalp},  and 
those associated in the same way  with $Q'=Q^{1/(p-1)}$ and $p'=p/(p-1)$.
Namely, if $1-1/p<\al\le1\le\albar$ correspond to $Q$ and $p$, and  
$1/p<\be\le1\le\bebar$ correspond to $Q'$ and $p'$, 
then we have 
\[
\be = \frac{\albar}{1+p(\albar-1)}
\quad \text{and} \quad
\bebar = \frac{\al}{1+p(\al-1)},
\]
and conversely, 
\[
\al = \frac{\bebar}{1+p'(\bebar-1)}
\quad \text{and} \quad
\albar = \frac{\be}{1+p'(\be-1)}.
\]
Note that $Q$, $p$ and $Q'$, $p'$  are dual in the sense that
$p=p'/(p'-1)$ and $Q=(Q')^{1/(p'-1)}$.

\end{remark}

\begin{proof}[Proof of Theorem~\ref{thm-regular-intro}]
Let $\be$ be as in Corollary~\ref{cor-beta-p1}.
By Theorem~\ref{thm-regular} and Corollary~\ref{cor-beta-p1}, 
the condition~\eqref{eq-div-wiener-sum}
is sufficient for regularity when
\[
\frac{p}{p-1}<s<\frac{1}{1-\be},
\]
which is equivalent to $1/\be<s/(s-1)<p$.
This in turn means that 
\[
0<\de=p-\frac{s}{s-1} < p-\frac{1}{\be} = \frac{p\be-1}{\be},
\]
and it follows that Theorem~\ref{thm-regular} is true for any exponent
\[
\frac{1}{\de} > \frac{\be}{p\be-1}.
\]
Finally, by Lemma~\ref{lem-al-be} we have
\begin{equation}   \label{eq-be-to-albar}
\frac{\be}{p\be-1} 
= \frac{\frac{\albar}{1+p(\albar-1)}}{\frac{p\albar}{1+p(\albar-1)}-1}
= \frac{\albar}{p-1},
\end{equation}
which finishes the proof.
\end{proof}

\begin{remark}
For $p=2$ it is easy to determine $\al$ in terms of $Q=Q(\al,p)$,
namely
\[
\alp=\bigl(Q \pm \sqrt{Q^2-Q}\bigr).
\]
For general $p>1$ this can be done numerically. 
However, noting that 
\[
\albar \le 1+p(\albar-1) < p\albar, 
\] 
we easily obtain the following estimate 
\[
Q^{1/(p-1)}\le\albar<(pQ)^{1/(p-1)}.
\]
This in particular shows that Theorems~\ref{thm-regular-intro}, 
\ref{thm-iter-pot-est} and~\ref{thm-regular} 
hold with the exponent 
\[
\frac1\de = \frac{(pQ)^{1/(p-1)}}{p-1}.
\]
This is more explicit than $1/\de=\albar/(p-1)+\eps$ but not sharp and the 
asymptotics as $Q\to1$ is not correct.
\end{remark}

\section{Sharpness of the capacitary estimates}
\label{sect-sharpness}

In this section we show that the exponent $1/\de$ in 
Theorems~\ref{thm-Olli} and~\ref{thm-iter-pot-est} is
sharp up to possibly the endpoint 
$p_1:=p_1(\frac{p}{p-1},Q^{1/p})$ for $s$.

For this, we shall use the following result from 
Bj\"orn--Bj\"orn~\cite[Theorem~5.1]{BB-power}.
Note, however, that the exponent $\al$ therein corresponds to $-\ga$ 
below and that, contrary to Theorem~\ref{thm-qmin-power-p>n}, 
here we consider $1<p<n$ and negative powers $|x|^{-\ga}$.

\begin{thm} \label{thm-qmin-power}
Let\/ $1<p<n$.
Then $|x|^{-\ga}$ is a quasiminimizer in\/ $B(0,1) \setm \{0\}\subset\R^n$
if and only if $\ga > n/p-1$ or $\ga=0$.
Moreover, if $\ga >n/p-1$, then
\begin{equation} \label{eq-def-Qalpn}
\biggl(\frac{p-1}{n-p}\biggr)^{p-1} \frac{\ga^p}{p\ga-(n-p)}
\end{equation}
is the best quasiminimizer constant for $|x|^{-\ga}$.
\end{thm}

\begin{example}
Let $1<p<n$ and $Q>1$ be fixed.
Corollary~\ref{cor-beta-p1} and~\eqref{eq-be-to-albar}
imply that 
\[
p-\frac{p_1}{p_1-1}=\frac{p-1}{\albar},
\] 
where $p_1=p_1(\frac{p}{p-1},Q^{1/p})$ is as in 
Corollary~\ref{cor-beta-p1} and
$\albar\ge1$ is the unique solution in $[1,\infty)$ of
\[
Q=\frac{\albar^{p}}{1+p(\albar-1)}.
\]
Let 
\[
\ga=\frac{\albar(n-p)}{p-1}.
\]
Then $\ga>n/p-1$ and Theorem~\ref{thm-qmin-power} implies that
$|x|^{-\ga}$ is a quasiminimizer in $B(0,1) \setm \{0\}\subset\R^n$
with the best quasiminimizer constant equal to
\begin{align*}
\biggl(\frac{p-1}{n-p}\biggr)^{p-1} \frac{\ga^p}{p\ga-(n-p)}
= \frac{\albar^p}{p\albar-(p-1)} =Q.
\end{align*}
For $x\in B(0,1)$ and sufficiently small $\eps>0$ let
\[
u_\eps(x)=\min\{\eps(|x|^{-\ga}-1),1\}
\]
and $E_\eps=B(0,\rho_\eps)$, where 
\[
\rho_\eps= \Bigl( \frac{\eps}{1+\eps} \Bigr)^{1/\ga}.
\]
Then $u_\eps$ is a $Q$-quasiminimizing potential for $E_\eps$ in $B(0,1)$.
Example~2.12 in Heinonen--Kilpel\"ainen--Martio~\cite{HeKiMa} shows
that for $r\ge 2\rho_\eps$,
\begin{equation}  \label{eq-cap-E-eps}
\cp(E_\eps,B(0,r)) \simeq \rho_\eps^{n-p} \simeq \eps^{(n-p)/\ga},
\end{equation}
where the comparison constants in $\simeq$ are independent of $\eps$ 
and $r$.
At the same time, for $B=B(0,\tfrac13)$ we have
\[
\inf_{2B} u_\eps = \eps ((\tfrac32)^\ga-1 ) 
< \eps (\tfrac32 )^\ga.
\]
Thus, by comparing this with~\eqref{eq-cap-E-eps} and letting $\eps\to0$ 
it follows that Theorem~\ref{thm-Olli} can only hold if
\begin{equation}  \label{eq-sharp-s-le-p1}
\de = p-\frac{s}{s-1} \le \frac{n-p}{\ga} 
= \frac{p-1}{\albar} = p-\frac{p_1}{p_1-1},
\end{equation}
i.e.\ for $s\le p_1$.

In Theorem~\ref{thm-iter-pot-est}, let $k\ge1$ be fixed but arbitrary.
Then for $B_j=B(0,r_j)$ with $r_j=3^{-(j+1)}$, $j=0,1,\ldots$, 
and $\eps^{1/\ga}\le r_{k+1}$, the left-hand side in~\eqref{eq-exp-est} is
\begin{equation}  \label{eq-est-inf-eps}
\inf_{B_{k+1}} u_\eps = \eps(r_{k+1}^{-\ga}-1)< \eps r_{k+1}^{-\ga}.
\end{equation}
At the same time, \eqref{eq-cap-E-eps} implies that 
the sum in the right-hand side of~\eqref{eq-exp-est} is
\[
\sum_{j=0}^{k} \biggl(
    \frac{\eps^{(n-p)/\ga}}{r_j^{n-p}}  
\biggr)^{1/\de}
= \eps^{(n-p)/\ga\de} \sum_{j=0}^{k} (r_j^{p-n})^{1/\de}
\simeq \biggl( \frac{\eps^{1/\ga}} {r_k} \biggr)^{(n-p)/\de},
\]
and hence 
\begin{align*}   
\inf_{B_{k+1}}u &\ge 1 - \exp \biggl( -c \sum_{j=0}^{k} \biggl(
    \frac{\capp(K\cap \itoverline{B}_{j+1},B_{j})}
      {r_j^{n-p}}  
\biggr)^{1/\de} \biggr) \\
&\ge 1- \biggl( 1-c' \biggl( \frac{\eps^{1/\ga}} {r_k} \biggr)^{(n-p)/\de}\biggr)
= c' \biggl( \frac{\eps^{1/\ga}} {r_k} \biggr)^{(n-p)/\de},
\end{align*}
where $c'$ (as well as the comparison constants in $\simeq$) 
can be chosen independently of $\eps$.
Comparing this with~\eqref{eq-est-inf-eps} and letting $\eps\to0$ shows
as in~\eqref{eq-sharp-s-le-p1} 
that Theorem~\ref{thm-iter-pot-est} can only hold for $s\le p_1$.
\end{example}

\end{document}